\documentclass[12pt,reqno]{amsart}

\usepackage[arrow,matrix,curve]{xy}

\usepackage[dvips]{graphicx} 

\usepackage{amssymb, latexsym, amsmath, amscd, array, hyperref
}

\usepackage{breakurl}   

\newtheorem{theorem}{Theorem}[section]
\newtheorem{lemma}[theorem]{Lemma}

\newtheorem{corollary}[theorem]{Corollary}

\theoremstyle{definition}
\newtheorem{definition}[theorem]{Definition}
\newtheorem{example}[theorem]{Example}

\numberwithin{equation}{section}

\newcommand\R {{\mathbb R}}

\newcommand\Z {{\mathbb Z}}

\newcommand\dist {\mathrm{dist}}

\author{Mikhail G. Katz}\address{M. Katz, Department of Mathematics,
Bar Ilan University, Ramat Gan 52900 Israel}
\email{katzmik@macs.biu.ac.il}

\begin{document}

\thispagestyle{empty}


\title{Torus cannot collapse to a segment}

\begin{abstract}
In earlier work, we analyzed the impossibility of co\-dimen\-sion-one
collapse for surfaces of negative Euler characteristic under the
condition of a lower bound for the Gaussian curvature.  Here we show
that, under similar conditions, the torus cannot collapse to a
segment.  Unlike the torus, the Klein bottle can collapse to a
segment; we show that in such a situation, the loops in a short basis
for homology must stay a uniform distance apart.
\end{abstract}

\maketitle

\section{Introduction}
\label{ss1}

In earlier work \cite{Ka19}, we analyzed the impossibility of
codimension-one collapse for surfaces of negative Euler
characteristic, if Gaussian curvature is bounded from below and the
diameter from above.  Here we show that, under similar conditions, the
torus cannot collapse to a segment, answering a question posed at MO
\cite{Al19}.  For a general framework for collapse of Riemannian
manifolds, see \cite{r5}, \cite{r6}, \cite[Section 8.25]{Gr99}.

Recall that an embedding~$N\subseteq M$ of Riemannian manifolds is
called \emph{strongly isometric} if the ambient distance and the
intrinsic distance in~$N$ coincide.

\begin{lemma}
\label{l11}
Let~$M$ be a closed surface with a Riemannian metric.  Consider a
homology basis
\begin{equation}
\label{e11}
[\alpha_1],\ldots,[\alpha_r] \in H_1(M;\Z_2)
\end{equation}
such that~$|\alpha_1|\leq\ldots\leq|\alpha_r|$ and the sum of the
lengths of the loops~$\alpha_i$ is minimal.  Then each of the
inclusions~$\alpha_1 \subseteq M, \ldots, \alpha_r\subseteq M$ is
strongly isometric.
\end{lemma}

\begin{proof}
We argue by contradiction.  Suppose that an inclusion~$\alpha_i
\subseteq M$ is not strongly isometric.  Then there are
points~$x,y\in\alpha_i$ partitioning~$\alpha_i$ into arcs~$\alpha_i'$,
$\alpha_i''$ such that a minimizing path~$\eta$ between~$x$ and~$y$ is
shorter than either arc.  Consider the loops~$\alpha_i'\cup\eta$
and~$\alpha_i''\cup\eta$ obtained as union of paths.  These loops are
shorter than~$\alpha_i$.  In the homology group, we have the relation
\begin{equation*}
[\alpha_i'\cup\eta] + [\alpha_i''\cup\eta] = [\alpha_i] \in H_1(M;\Z_2),
\end{equation*}
where orientations can be ignored since we are working with
$\Z_2$-coeffi\-cients.  Exactly one of the new loops,
say~$\alpha_i'\cup\eta$, lies outside the hyperplane spanned by the
classes of~$(\alpha_1,\ldots,\alpha_{i-1},
\alpha_{i+1},\ldots,\alpha_r)$.  Since we are working with
$\Z_2$-coefficients, the family of loops
\[
(\alpha_1,\ldots,\alpha_{i-1},{\alpha_i'\cup\eta},\alpha_{i+1},\ldots,\alpha_r)
\]
is a basis for homology.  By construction, the sum of the lengths of
this new family of loops is smaller than the sum of the lengths of the
original family~\eqref{e11}, contradicting the definition of that
family and proving the lemma.
\end{proof}

\begin{corollary}
\label{l12}
When~$M=T^2$ is the~$2$-torus, the loops~$\alpha=\alpha_1$ and
$\beta=\alpha_2$ represent elements of a basis of~$H_1(T^2;\Z)$ and
meet at a single point.
\end{corollary}

\begin{proof}
By definition,
\begin{equation}
\label{e12}
|\alpha| \leq |\beta|.
\end{equation}
Since the classes~$[\alpha], [\beta]\in H_1(T^2;\Z_2)$ are distinct,
they have nonzero algebraic intersection and therefore nontrivial
geometric intersection.  To prove that~$\alpha$ and~$\beta$ also
represent a basis of~$H_1(T^2;\Z)$, it suffices to show that the
intersection~$\alpha\cap\beta$ consists of a single point.

We will argue by contradiction.  Suppose there are two distinct points
$x,y$ in~$\alpha\cap\beta$.  Let~$\alpha_0$ be a shortest arc of
$\alpha$ connecting~$x$ and~$y$, and~$\beta_0$ a shortest arc of
$\beta$ connecting~$x$ and~$y$.

Suppose~$\alpha_0\sim\beta_0$ (i.e., the union~$\alpha_0\cup\beta_0$
is~$\Z_2$-nullhomologous).  Then either~$\alpha$ or~$\beta$ can be
shortened by replacing its arc between~$x$ and~$y$ by the
corresponding arc of the other loop, and rounding off the corners.
This contradicts the minimality of the basis~$[\alpha],[\beta]$.

Hence~$\alpha_0\cup\beta_0$ represents a nonzero class
in~$H_1(T^2;\Z_2)$.  By construction
$|\alpha_0\cup\beta_0|\leq|\beta|$.  Let~$\delta$ be a minimizing loop
in the class~$[\alpha_0\cup\beta_0]\in H_1(T^2;\Z_0)$.  Then
$|\delta|<|\beta|$.  Now \eqref{e12} implies that 
\begin{equation}
\label{e13}
\delta\in [\alpha]=[\alpha_0\cup\beta_0],
\end{equation}
and therefore~$|\delta|=|\alpha|$.

Consider the complementary subarcs~$\alpha_1=\alpha\setminus\alpha_0$
and~$\beta_1=\beta\setminus\beta_0$.  By~\eqref{e13} we have
$\alpha_1\sim\beta_0$.  If~$|\beta_0|<|\alpha_1|$ then~$\alpha$ can be
shortened to~$\alpha_0\cup\beta_0$, contradicting the minimality of
the loop~$\alpha$ in its~$\Z_2$-homology class; if
$|\beta_0|>|\alpha_1|$ then~$\beta$ can be shortened to
$\alpha_1\cup\beta_1$, contradicting the minimality of the loop
$\beta$ in its~$\Z_2$-homology class (the case of equality is treated
by rounding off corners as before).  The contradiction proves that
$\alpha$ and~$\beta$ generate a basis for integer homology.
\end{proof}

Note that the conclusion of Corollary~\ref{l12} may fail for the Klein
bottle; cf.\ Example~\ref{e33}.

\begin{definition}
Let~$p=\alpha\cap\beta\in T^2$ be the intersection point of the two
loops.
\end{definition}

Thus, the loops~$\alpha$ and~$\beta$ represent generators of
$\pi_1(T^2,p)\simeq\Z^2$.

\begin{definition}
\label{dd15}
Choosing a parametrisation for each of the loops based at~$p$, we form
the \emph{commutator
loop}~$\alpha\circ\beta\circ\alpha^{-1}\circ\beta^{-1}$ based at~$p$.
\end{definition}

\section
{Circle noncollapse}
\label{s2}

\begin{lemma}
Suppose there exists a metric~$d$ on the disjoint union of a
circle~$S$ of length~$L$ and a segment~$I$ such that both~$S$ and~$I$
are strongly isometrically embedded.  Assume that~$S$ is contained in
a~$\nu$-neighborhood of~$I$. Then~$\nu > \frac{L}{16}$.
\end{lemma}

\begin{proof}
Let~$a, b \in S$ be two antipodal points on~$S$.  Let~$J_1$ and~$J_2$
be the two arcs of~$S$ connecting~$a$ and~$b$.  Clearly
$|J_1|=|J_2|=\frac{L}{2}$.  Both~$J_1$ and~$J_2$ are strongly
isometrically embedded in~$S \cup I$.  Choose~$a',b'\in I$ such
that~$d(a, a')<\nu$ and~$d(b, b') < \nu$.  Clearly~$|d(a', b') -
\frac{L}{2}| < 2\nu$.  Let~$c_1$ and~$c_2$ be the midpoints of the
arcs~$J_1$ and~$J_2$, respectively.  One can find points~$c'_1, c'_2
\in I$ such that~$d(c_1, c'_1) < \nu$ and~$d(c_2, c'_2) < \nu$.  We
have
\[ 
|d(a', c'_i) - \tfrac{L}{4}|<2\nu,\;|d(b',c'_i)-\tfrac{L}4|<2\nu\text{
 for } i = 1, 2.
\] 
It follows that
\[
|d(a',c'_i)-\tfrac12 d(a',b')|<3\nu,\;|d(b',c'_i)-\tfrac12 d(a',b')| <
  3\nu \text{ for }, i = 1,2.
\] 
Let~$c'$ be the midpoint of the segment~$[a', b']$.
Then~$d(c',c'_i)<3\nu$ for~$i = 1, 2$.  Hence~$\frac{L}2 = d(c_1,c_2)
\leq d(c1, c'_1) + d(c'_1, c') + d(c', c'_2) + d(c_2, c'_2)< 8\nu$.
\end{proof}

With little extra effort, one can prove a near-optimal bound of the
same type; see Lemma~\ref{l21}.

\begin{definition}
Let~$S^1$ be the Riemannian circle of length~$2\pi$. 
\end{definition}

Recall \cite{Ka83} that the filling radius of~$S^1$ equals
\begin{enumerate}
\item
$\frac16$ of its length, 
\item
namely~$\frac{\pi}{3}$,
\item
half the (spherical) diameter of the equilateral triangle inscribed
in~$S^1$.
\end{enumerate}
For recent applications of the filling radius, see \cite{Li20}.

\begin{lemma}
\label{l21}
Let~$I\subseteq\R$ be a segment.  The circle~$S^1$ cannot be strongly
isometrically embedded in a~$\frac{\pi}{3}$-neighborhood of~$I$ in any
metric space~$Z$.
\end{lemma}

\begin{proof}
Suppose~$S^1$ is contained in an open~$\nu$-neighborhood of~$I$ in a
metric space~$(Z,d_Z)$.  It will be convenient to use the language of
infinitesimals, infinite proximity, etc.  These terms can either be
understood as shorthand for~$\epsilon, \delta$ arguments, or can be
interpreted literally in terms of a formalisation using an
infinitesimal-enriched continuum as in \cite{15d} following
\cite{Ro66}.  We view~$S^1$ as an infinilateral polygon.  We send each
vertex of the polygon to a nearest point of~$I\subseteq Z$.  We extend
by linearity to obtain a piecewise linear map~$f\colon S^1\to I$.

By the Borsuk--Ulam theorem, there is a pair of opposite points of
$S^1$ with the same image in~$I$ under~$f$.  Consider the
infinitesimal sides~$aa'\subseteq S^1$ and~$bb'\subseteq S^1$
containing these opposite points of~$S^1$.  Then their images overlap,
i.e.,
\begin{equation}
\label{e21b}
f(aa') \cap f(bb')\not=\emptyset.
\end{equation}
By strong isometry and triangle inequality, we have
\begin{equation}
\label{e21}
d_Z(a,b)\approx \pi,
\end{equation}
where~$\approx$ denotes the relation of infinite proximity.
Let~$AA'=f(aa')$ and~$BB'=f(bb')$.  Then~$|AA'| < 2\nu$
and~$|BB'|<2\nu$.  The overlap condition~\eqref{e21b} means that the
segments~$AA', BB' \subseteq I$ overlap.  Hence one of the images,
say~$BB'$, must contain a vertex, say~$A$, of the image~$AA'$ of the
other side, so that~$A\in BB'$.  We can assume that~$B\in I$ is the
endpoint of the side~$BB'$ closest to the point~$A$.
Then~$d_Z(A,B)\leq \nu$.  Thus~$d_Z(a,b) \leq d_Z(a, A) + d_Z(A,B) +
d_Z (B,b) \leq \nu+\nu+\nu= 3\nu$.  Finally the relation \eqref{e21}
implies~$\nu\geq \frac{\pi}{3}$.
\end{proof}

\section{Separating geodesic loop}

We choose loops~$\alpha,\beta\subseteq T^2$ as in Section~\ref{ss1},
and let~$p=\alpha\cap\beta$.  A loop in~$T^2$ is called
\emph{separating} if its complement is~$T^2$ has two connected
components.

\begin{lemma}
\label{l31}
Suppose a Riemannian torus~$T^2$ is at Gromov--Hausdorff distance at
most~$\nu$ from a segment~$I$.  If~$|I|> 52 \nu$ then there exists an
embedded separating geodesic loop~$\lambda_p$ in~$T^2$ based at the
point~$p$ (possibly non-smooth at~$p$) and of length less
than~$24\nu$.
\end{lemma}

\begin{proof}
By Lemma~\ref{l11}, the embedding of the loops~$\alpha,\beta\subseteq
T^2$,~$|\alpha|\leq|\beta|$, is strongly isometric.  By
Lemma~\ref{l21}, we have
$\nu\geq\frac{\pi}{3}\frac{|\beta|}{2\pi}=\frac{|\beta|}{6}$, and
therefore
\begin{equation}
\label{e31c}
|\alpha|\leq|\beta|\leq6\nu.
\end{equation}
Once we have such a ``short'' pair of loops on~$T^2$, we have some
indication that~$T^2$ ``looks like'' a small flat torus with a
Gromovian ``long finger;'' cf.\ \cite[pp.\;297, 321, 324, 327]{Gr96},
\cite[p.\;243]{Gr99}.

Suppose~$T^2$ and~$I$ are at GH distance at most~$\nu$ in an ambient
metric space~$Z$.  Choose a point~$P\in I$ nearest to~$p\in
T^2\subseteq Z$.  Let~$Q\in I$ be an endpoint of~$I$ furthest away
from~$P$, so that~$|PQ| >26\nu$ by hypothesis.  Let~$q\in T^2$ be a
point nearest to~$Q$.  The point ~$q\in T^2$ can be thought of as the
end of the long finger.  Then
\begin{equation}
\label{e31b}
\dist(p,q)>24\nu.
\end{equation}
Consider the commutator
loop
\begin{equation}
\label{e30}
\sigma_p=\alpha\circ\beta\circ\alpha^{-1}\circ\beta^{-1}
\end{equation}
based at~$p=\alpha\cap\beta$ as in Definition~\ref{dd15}.  By
\eqref{e31c}, the length of the commutator loop is at most
\begin{equation}
\label{e32}
|\sigma_p| \leq 4|\beta|\leq 24\nu.
\end{equation}
It follows from \eqref{e31b} that~$q$ cannot be reached by the
commutator loop~$\sigma_p$, or by any shorter loop based at~$p$.

We now cut the torus open along~$\alpha$ and~$\beta$.  This produces a
convex ``parallelogram''~$\Pi=\Pi_{\alpha,\beta}$.  In the
complement~$\Pi\setminus\{q\}$, consider a shortest noncontractible
loop~$\lambda$ based at a corner of~$\partial{\Pi}$.  In particular,
$|\lambda|<|\sigma_p|$.  Then~$\lambda$ is simple (see below), remains
far away from the missing point~$q$, and stays in the interior
of~$\Pi_{\alpha,\beta}$ (except for the basepoint) due to the
convexity of~$\Pi_{\alpha,\beta}$.  Such a loop projects to a
separating embedded loop~$\lambda_p\subseteq T^2$.

To construct such a loop~$\lambda$, consider the universal cover~$U$
of~$\Pi\setminus\{q\}$.  Let~$P$ be a lift of~$p$ to~$U$.  Let~$P'\in
U$ be a nearest point to~$P$ in the orbit of~$P$ under the action
of~$\pi_1\big(\Pi\setminus\{q\}, p\big)\simeq\Z$ on~$U$.  Let
$\tilde\lambda\subseteq U$ by a minimizing path from~$P$ to~$P'$.  Let
$\lambda\subseteq\Pi\setminus\{q\}$ be the geodesic loop based at~$p$
obtained as the projection of~$\tilde \lambda$.  Suppose~$\lambda$ is
not simple.  Then there is a translated path~$RR'$
meeting~$\tilde\lambda=PP'$ transversely at an interior point
$X=PP'\cap RR'$.  Of the~$4$ points~$P,P',R,R'$ consider the nearest
one to~$X$.  We can assume without loss of generality that this point
is~$P$.  

One of the points~$R,R'$ is necessarily distinct from~$P$.  We can
assume without loss of generality that this point is~$R$.  Consider
the broken geodesic path~$PX\cup XR$.  Its projection to $\Pi\setminus
\{q\}$ is noncontractible, since~$P\not=R$.  By construction,~$|PX\cup
XR|\leq|\lambda|$.  By rounding off the corner at~$X$, we obtain a
path connecting~$P$ to~$R$ of length strictly shorter
than~$|\lambda|$.  This contradicts the choice of~$\tilde \lambda$ and
proves that~$\lambda$ is a noncontractible simple loop of
$\Pi\setminus\{q\}$.

Alternatively, one could start with the boundary~$\partial\Pi$ viewed
as a loop based at a corner of~$\Pi_{\alpha,\beta}$, and apply
Sabourau's modification \cite{Sa04} of the curve-shortening process of
Hass--Scott \cite{Ha94}.  The process preserves embeddedness and
yields an embedded geodesic loop.
\end{proof}

Note that the loop~$\lambda_p$ of Lemma~\ref{l31} is not necessarily a
closed geodesic of~$T^2$, as it may be nonsmooth at~$p$.

\section{Torus non-collapse}

We now show that a~$2$-torus cannot collapse to a segment under the
constraint \mbox{$K\geq-1$}.

\begin{theorem}
There is no sequence of~$2$-torii equipped with smooth Riemannian
metrics with Gauss curvature bounded from below, which converges in
the Gromov--Hausdorff sense to a (non-degenerate) segment.
\end{theorem}

\begin{proof}
Suppose there is such a sequence of metrics.  For metrics sufficiently
close to the segment, the hypothesis of Lemma~\ref{l31} is satisfied.
Then the loop~$\lambda_p$ of Lemma~\ref{l31} separates~$T^2$ into a
disk (containing~$q$) and a torus-with-boundary~$M$ in the homotopy
type of~$S^1\vee S^1$.  The total geodesic curvature of~$\lambda_p$ is
concentrated at the point~$p$ and bounded by~$\pi$.  We apply the
Gauss--Bonnet theorem to the surface~$M$ to obtain
\begin{equation}
\label{e31}
\begin{aligned}
\int_M K\,dA =2\pi\chi(M)-\int_{\lambda_p}k_g(s)\;ds
=-2\pi-\int_{\lambda_p}k_g(s)\;ds \leq -\pi.
\end{aligned}
\end{equation}
To show that the curvature has to be very negative somewhere
in~$M\subseteq T^2$, it therefore suffices to show that the area
of~$T^2$ must tend to zero as the torus collapses to~$I$.  Indeed,
consider a fine partition of~$I$ with step~$\nu$.  For each partition
point~$P_k\in I$, choose a nearest point~$p_k\in T^2$.  If~$T^2$ is
contained in a~$\nu$-neighborhood of~$I$, then by the triangle
inequality, disks of radius~$3\nu$ centered at
the~$n=\left\lfloor\frac{|I|}{\nu}\right\rfloor$ points~$p_k$ cover
the torus:
\[
T^2=\cup_{k=1}^n B_{3\nu}(p_k).
\]
The lower bound~$K\geq-1$ implies an upper bound of type~$C\nu^2$ on
the area of the disks, where~$C>\pi$ can be chosen as close to~$\pi$
as one wishes by choosing~$\nu$ sufficiently small.  It follows that
the area of the torus collapsing to the segment~$I$ is bounded by
$\frac{|I|}{\nu} C\nu^2 = |I|C\nu$ and hence must tend to zero
with~$\nu$.  Then the estimate~\eqref{e31} implies that the Gaussian
curvature must be very negative somewhere in~$M\subseteq T^2$,
contradicting the hypothesis of collapse of~$T^2$ to~$I$ constrained
by the condition~$K\geq-1$.  The contradiction proves the theorem.
\end{proof}

\section{Klein bottle collapse}

\begin{example}
\label{e33}
Unlike the torus, the Klein bottle can collapse to a segment.
Let~$\epsilon>0$.  Consider the maps~$f\colon
\R^2\to\R^2$,~$(x,y)\mapsto(x+2,y)$ and~$g\colon
\R^2\to\R^2$,~$(x,y)\mapsto(-x,y+\epsilon)$.  As a fundamental domain
one can choose the rectangle~$[-1,1)\times[0,\epsilon)$.  Then the
flat Klein bottle~$K=\R^2/\langle f;g\rangle$ collapses to the
segment~$[0,1]\times\{0\}$ when~$\epsilon\to0$.  Here the shortest
basis for~$H_1(K;\Z_2)\simeq\Z_2^2$ is given by the disjoint loops
parametrized by~$\{0\}\times [0,\epsilon)$
and~$\{1\}\times[0,\epsilon)$.
\end{example}

\begin{theorem}
If the Klein bottle~$K$ collapses to a segment then there is a uniform
positive lower bound for the minimal distance between the loops of a
short basis for~$H_1(K;\Z_2)$.
\end{theorem}

\begin{proof}
Let~$\alpha,\beta$ be such a short basis.  If~$K$ collapses to a
segment then~$|\alpha|\to0$ and~$|\beta|\to0$ as before, based on
Section~\ref{s2}.  If there is a subsequence such that
$\alpha\cap\beta\not=\emptyset$ then we cut~$K$ open along~$\alpha$
and~$\beta$ to form a convex parallelogram~$\Pi_{\alpha,\beta}$ as in
the case of the torus (the identification along the boundary is
different for~$K$).  We find a distant point~$q\in\Pi_{\alpha,\beta}$,
consider a shortest noncontractible
loop~$\lambda\subseteq\Pi\setminus\{q\}$ based at a corner
of~$\Pi_{\alpha,\beta}$, project the loop to~$K$, and take the
connected component~$M\subseteq K\setminus\lambda$ not homeomorphic to
a disk.  Then~$M$ is a one-holed Klein bottle in the homotopy type of
$S^1\vee S^1$.  We then apply the Gauss--Bonnet formula to~$M$ as
in~\eqref{e31}, to show that collapse cannot occur in this case.

Thus we may assume that~$\alpha$ and~$\beta$ are disjoint.
Let~$\gamma$ be a shortest loop in the class~$[\alpha]+[\beta]\in
H_1(K;\Z_2)$.  Note that the class~$[\gamma]$ is a characteristic
vector for the~$\Z_2$-intersection form, i.e.,~$[\gamma]\cap x = x\cap
x$ for all~$x\in H_1(K;\Z_2)$.  In
particular,~$\alpha\cap\gamma\not=\emptyset$.

Suppose the minimal distance between~$\alpha$ and~$\beta$ tends to
zero.  Then the length~$|\gamma|$ tends to~$0$, as well.  We cut~$K$
open along~$\alpha$ and~$\gamma$ to form a convex
parallelogram~$\Pi=\Pi_{\alpha,\gamma}$ as before.  We find a distant
point~$q\in\Pi_{\alpha,\gamma}$, consider a shortest noncontractible
loop~$\lambda\subseteq\Pi\setminus\{q\}$ based at a corner
of~$\Pi_{\alpha,\gamma}$, take the connected component~$M\subseteq
K\setminus\lambda$ not homeomorphic to a disk, and apply the
Gauss--Bonnet formula to~$M$ as in \eqref{e31}, to show that collapse
cannot occur.  Therefore the minimal distance between~$\alpha$
and~$\beta$ cannot tend to zero.
\end{proof}

Recently Zamora \cite{Za20} generalized the non-collapsing result to
tori of arbitrary dimension.

\section*{Acknowledgments}  
We are grateful to Tahl Nowik, St\'ephane Sabourau, and the anonymous
referee for helpful comments on earlier versions of the article.

\end{document}